\documentclass[11pt]{amsart}
\usepackage{amssymb,comment}
\usepackage{amsmath}
\usepackage{bbm}
\usepackage{graphicx} 
\usepackage[english]{babel} 
\usepackage{float}
\usepackage{mathtools}
\usepackage{hyperref}
\usepackage[noadjust]{cite}
\usepackage[margin=3.2cm]{geometry}
\hypersetup{ 
    colorlinks=true,       % false: boxed links; true: colored links
    linkcolor=blue,          % color of internal links
    citecolor=blue,        % color of links to bibliography
    filecolor=blue,      % color of file links
    urlcolor=blue           % color of external links
}

\usepackage{array}

\usepackage{pstricks} % Compile using XeLatex, View PDF
\usepackage{pst-all}
\usepackage{epsfig}
\usepackage{pst-grad} % For gradients
\usepackage{pst-plot} % For axes
\usepackage[space]{grffile} % For spaces in paths
\usepackage{etoolbox} % For spaces in paths
\usepackage{color} %for color text
\makeatletter % For spaces in paths
\patchcmd\Gread@eps{\@inputcheck#1 }{\@inputcheck"#1"\relax}{}{}
\makeatother

\usepackage[all]{xy}
\usepackage{mathrsfs}
\usepackage{graphics}
\usepackage{amsthm}

\theoremstyle{plain}\newtheorem{theorem}{Theorem}[section]\newtheorem{Theorem}{Theorem}\newtheorem{proposition}[theorem]{Proposition}\newtheorem{lemma}[theorem]{Lemma}\newtheorem{corollary}[theorem]{Corollary}

\theoremstyle{definition}\newtheorem{remark}[theorem]{Remark}%\newtheorem{construction}[subsection]{Construction}

\def\Q{\mathbb{Q}}\def\C{\mathbb{C}}\def\Z{\mathbb{Z}}\def\R{\mathbb{R}}\def\ZZ2{\mathbb{\Z/ 2\Z}}
\def\sb{\subset}\def\ot{\otimes}\def\t{\times}\def\sm{\setminus}
\def\a{\alpha}\def\b{\beta}\def\e{\epsilon}\def\s{\sigma}\def\De{\Delta}\def\la{\lambda}\def\la{\lambda}\def\p{\partial}
\def\wt{\widetilde}
\def\sl2{\mathfrak{sl}_2}
\def\su2{\mathfrak{su}(2)}
\def\id{\text{id}}

\def\deg{\text{deg}}

\def\UHq{U^H_q(\sl2)}\def\Usl2{U_q(\sl2)}\def\usl2{\wt{U}_q(\sl2)}\def\uq{\mathfrak{u}_q(\sl2)}

\def\Rep{\text{Rep}}

\def\CC{\mathcal{C}}

\def\ZZ{\mathcal{Z}}

\def\CC{\mathcal{C}}

\def\HH{H=\{\Ha\}_{\a\in G}}

\def\Z{\mathbb{Z}}
\def\C{\mathbb{C}}
\def\ot{\otimes}\def\id{\text{id}}

\def\sl{\mathfrak{sl}}

\def\TT{\overline{\Theta}}
\def\qa{q}
\def\CC{\mathcal{C}}

\def\rev{\overrightarrow{\text{ev}}}

\def\rcoev{\overrightarrow{\text{coev}}}
\def\lev{\overleftarrow{\text{ev}}}

\def\sl{\mathfrak{sl}_2}

\def\uq{\mathfrak{u}_q(\sl)}
\def\HH{\mathcal{H}}
\def\ADO{N}
\def\SS{S}

\title{Non-semisimple $\sl$ quantum invariants of fibred links}
\author{Daniel López Neumann and Roland van der Veen}
\email{dlopezne@indiana.edu, r.i.van.der.veen@rug.nl}

\begin{document}

\maketitle

\begin{comment}

We prove that the Akutsu-Deguchi-Ohtsuki (ADO) invariant of a fibred link in $S^3$ has maximal degree allowed by its genus bound and that its top coefficient is a power of $q=e^{\frac{\pi i}{p}}$ determined by the Hopf invariant of the plane field of $S^3$ associated to the fiber surface. Our proof is based on the genus bounds for ADO invariants established in our previous work together with a theorem of Giroux-Goodman stating that fiber surfaces in the three-sphere can be obtained from a disk by plumbing/deplumbing Hopf bands

\end{comment}

\begin{abstract} 

 The Akutsu-Deguchi-Ohtsuki (ADO) invariants are the most studied quantum link invariants coming from a non-semisimple tensor category. We show that, for fibered links in $S^3$, the degree of the ADO invariant is determined by the genus and the top coefficient is a root of unity. More precisely, we prove that the top coefficient is determined by the Hopf invariant of the plane field of $S^3$ associated to the fiber surface. Our proof is based on the genus bounds established in our previous work, together with a theorem of Giroux-Goodman stating that fiber surfaces in the three-sphere can be obtained from a disk by plumbing/deplumbing Hopf bands. \end{abstract}

\section{Introduction}

A big open problem in quantum topology is to find connections between quantum invariants and geometric topology. This problem has been mostly studied for quantum invariants coming from semisimple tensor categories, such as quantum groups at generic parameter. For the colored Jones polynomials (which come from $U_q(\sl)$ at generic $q$), there are several conjectures in this direction \cite{Kashaev:hyperbolic, Garoufalidis:character, Witten:quantum}. For other tensor categories this problem is largely unexplored.
\medskip

%Say something lie ``but few general results \cite{Ohtsuki:2loop, Eisermann:ribbon}". Appropriate references?

In our previous work \cite{LNV:genus, LNV:genus-unrolled}, we showed that quantum link invariants coming from certain non-semisimple tensor categories 
structurally differ from their semisimple counterparts. %DIFFERENT FROM WHAT? From the generic q invariants, with respect to the above question.
For instance, if $N_p(L,x)$ is the Akutsu-Deguchi-Ohtsuki (ADO) invariant of a link $L\sb S^3$ \cite{ADO, Murakami:ADO} (which comes from $U_q(\sl)$ at a $2p$-th root of unity $q$), we showed that:
\begin{align}
    \label{eq: genus bound for ADO}
    \deg_x \ADO_p(L,x)\leq (2g(L)+s-1)(p-1)
\end{align}
where $g(L)$ is the minimal genus of a connected Seifert surface of $L$ and $s$ is the number of components of $L$. This generalizes the well-known genus bound for the Alexander polynomial, which is the case $p=2$ by \cite{Murakami:Alexander}. No similar result is known for the colored Jones polynomials. Note that the ADO invariants are at the source of various recent developments in quantum topology, such as non-semisimple TQFTs \cite{BCGP}, the $q$-series invariants predicted by physicists \cite{GM:two-variable, Gukov:Coulomb} and vertex operator algebras \cite{CDGG:QFT}. 
\medskip

%For an integer $p\geq 2$, the one-variable Akutsu-Deguchi-Ohtsuki (ADO) invariant of a link $L\sb S^3$ is an isotopy invariant $\ADO_p(L)\in \Z[e^{\frac{\pi i}{p}}][x^{\pm \frac{1}{2}}]$ defined through the representation theory of the quantum group $U_q(\sl)$ at the $2p$-th root of unity $q=e^{\frac{\pi i}{p}}$ \cite{ADO, Murakami:ADO}. %More precisely, it is defined by applying the Reshetikhin-Turaev construction \cite{RT1} to a $p$-dimensional $U_q(\sl)$-module with arbitrary highest weight $\la\in \C$ and letting $x=q^{2\la+2}$. %These invariants are at the source of various recent developments in quantum topology, such as non-semisimple TQFTs \cite{BCGP}, $q$-series invariants of knots and 3-manifolds \cite{GM:two-variable, Gukov:Coulomb} and vertex operator algebras \cite{CDGG:QFT}. The ADO invariants are often referred as ``non-semisimple" quantum invariants, since the tensor categories involved are non-semisimple.

In this paper, we investigate the ADO invariants for fibred links. Recall that a link $L\sb S^3$ is {\em fibred} if there exists a fibration $S^3\sm L\to S^1$ in which each fiber is a Seifert surface. Generalizing the well known result that the Alexander polynomial of a fibered knot is monic and the genus bound is sharp we prove the following:
\def\rk{\text{rk}}

\medskip
\begin{Theorem}
\label{Thm: main fibred}
Let $L\sb S^3$ be a fibred $s$-component link with fiber surface $F$. Then, for any $p\geq 2$, $\ADO_p(L,x)$ has degree $(2g(L)+s-1)(p-1)$ and its top coefficient is given by $$(-q)^{2g(L)+s-1-2\la(\xi_F)}$$ where $\la(\xi_F)$ is the Hopf invariant of the plane field $\xi_F$ of $S^3$ associated to $F$.

\end{Theorem}
\medskip

To the knowledge of the authors, Theorem \ref{Thm: main fibred} is the first theorem relating quantum link invariants to fibredness. No relation seems to be known for the ``semisimple" link polynomials, such as colored Jones, HOMFLY and Kauffman polynomials, or even their categorifications such as Khovanov homology. It is worth noting that certain obstructions to fibredness can be obtained from the Witten-Reshetikhin-Turaev TQFT \cite{Gilmer:cohomology, AB:turaev-viro}.
\medskip

%the 2-loop expansion of the Kontsevich invariant, which has a genus bound \cite{Ohtsuki:2loop}

Theorem \ref{Thm: main fibred} follows from Theorem \ref{Theorem: top coefficient of ADO after plumbing} below, which studies the effect of Hopf plumbing on (the top term of) ADO invariants, together with a result of Giroux-Goodman \cite{GG:fibred} that states that all fiber surfaces in $S^3$ are obtained from a disk by plumbing/deplumbing Hopf bands (see Subsection \ref{subs: plumbing Hopf} and Figure \ref{fig:prettyplumbing}). The proof of Theorem \ref{Theorem: top coefficient of ADO after plumbing} is purely algebraic and relies on the genus bounds as established in \cite{LNV:genus-unrolled}. Note that, in contrast, the proof of the Giroux-Goodman theorem of \cite{GG:fibred} relies on contact geometry techniques.
\medskip

Our method turns out to apply to other non-semisimple quantum invariants as well, such as the Links-Gould invariant or analogues of ADO invariants for other quantized Lie algebras at roots of unity (this will be carried out in a separate paper). Indeed, all these invariants are defined in a similar way (using Verma modules of arbitrary complex highest weight) and they all satisfy a genus bound \cite{LNV:genus-unrolled, KT:Links-Gould}. The key point, and the main difference with the generic $q$ case, is the $\C^*$-graded structure of the tensor categories behind the construction of these invariants (or the ``equivariant" nature of such categories, as in \cite[Theorem 2]{LNV:genus-unrolled}).

%\textcolor{blue}{In the abstract we say our theorem follows from our previous work on genus and GG. We should say that genus bounds are necessary for Thm 2.}

%follows from the following more general result (see Subsection \ref{subs: plumbing Hopf} for the definition of plumbing):%It is possible that our theorem should follow from the (non-semisimple) TQFT for ADO invariants \cite{BDGG:ADO}, but we consider the proof here is much more elementary (except for the fact that the Giroux-Goodman theorem relies on contact geometry).

\begin{figure}[h]
\label{fig:prettyplumbing}
\begin{center}
    \includegraphics[width=6cm]{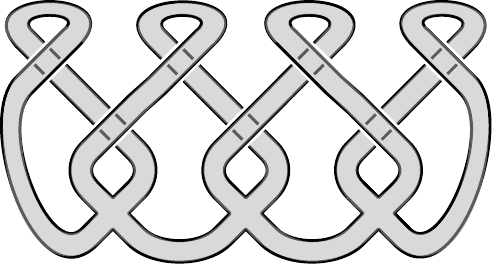}
\end{center}
\caption{The fiber surface of the fibered knot $6_3$ from plumbing two positive Hopf bands (left) and two negative ones (right), see also Section \ref{section: computations}.}
\end{figure}

For certain classes of fibred links one can say more about how they are obtained by plumbing/deplumbing Hopf bands. These classes include: fibred strongly-quasipositive links \cite{Rudolph:quasipositive-plumbing, Rudolph:quasipositive3-characterization}, homogeneous braid closures (which are fibred by \cite{Stallings:constructions-fibred}) and, in particular, positive braid closures. See Subsection \ref{subs: corollaries} for the corresponding corollaries for ADO invariants.
\medskip

For an illustration of our theorem on knots with few crossings, see Table \ref{table: N4 of small knots} in Section \ref{section: computations}. See also Table \ref{table: 12-crossing knots where Alex fails} which shows that $\ADO_4$ detects the non-fibredness of all but one (the knot $12n57$) of the non-fibered 12-crossing knots that have monic Alexander polynomial of maximal degree.
\medskip

 Of course, it is well-known that Floer-theoretic invariants {\em detect} fibredness \cite{Ni:fibred, KM:knots-sutures-excision} (knot Floer homology also detects the Hopf invariant \cite{OS:contact}), though these invariants are of a completely different nature to the ones considered here. It is also known that twisted Alexander polynomials detect fibredness \cite{FV:fibered}.

\subsection{Plan of the paper} Section \ref{section: fibred links and plumbing} contains the basic topological notions of the paper and in Section \ref{section: ADO invariants} we recall the definition of the ADO invariants. Section \ref{section: proof of main thm} is devoted to the proof of Theorem \ref{Thm: main fibred}, which is carried out in Subsection \ref{subs: PROOF OF MAIN}, and in Subsection \ref{subs: corollaries} we deduce some corollaries. In Section \ref{section: computations} we present computations and remarks.

   \subsection{Acknowledgments} The authors would like to thank Sebastian Baader, Christian Blanchet and Paul Kirk for useful conversations.

\section{Fibred links and fiber surfaces}
\label{section: fibred links and plumbing}

\subsection{Fibred links} By a {\em fibred link} in $S^3$ we mean an oriented link $L$ together with a fibration $S^3\sm L\to S^1$ in which the closure of each fiber is a Seifert surface for $L$. %Note that this definition depends on the orientation of the link, for instance, a $(2,2n)$-torus link is fibred, but reversing the orientation of a strand is no longer fibred if $n\neq \pm 1$ (of course, for knots, being fibred is independent of the orientation). 
If $L$ is fibred, then the fibration (and hence the fiber surface) is unique up to isotopy, hence being fibred is a property of links and not an additional structure. Moreover, it is well-known that the fiber surface of a fibred link is the unique Seifert surface of minimal genus. %If $L$ is a fibred link, then $S^3\sm N(L)$ is fiber-preserving homeomorphic to $F\t [0,1]$ with $F\t 0$ glued to $F\t 1$ via a diffeomorphism $f:F\t 1\to F\t 0$ with $f|_{\p F}=\id_{\p F}$.
\medskip

Let $H_+$ and $H_-$ be the positive and negative Hopf links respectively, these are fibred links. Let $A^+$ (resp. $A^-$) be an unknotted, oriented annulus embedded in $S^3$ with a negative (resp. positive) full twist. Then $\p A^{\pm}=H_{\pm}$ as oriented links, $A^{\pm}$ are the fiber surfaces for the respective Hopf links. We call $A^+$ (resp. $A^-$) a  positive (resp. negative) {\em Hopf band}.

\begin{figure}[h]
\label{fig:plumbing}
\begin{center}
    \includegraphics[width=5cm]{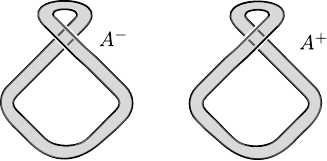}
\end{center}
\caption{The negative and positive annuli $A^\pm$.}
\end{figure}

\subsection{Plumbing Hopf bands}\label{subs: plumbing Hopf} 
Let $F \subset \SS^3$ be an oriented compact surface with boundary and
$\a \subset F$ a properly embedded simple arc. We say that a compact surface $F' \subset \SS^3$ is obtained from $F$ by {\em plumbing} a Hopf band $A^{\pm}$ along $\a$ if $F' = F \cup A^{\pm}$ in such a way that:
\begin{enumerate}
\item the intersection $A^{\pm} \cap F$ is a tubular neighborhood of $\a$ in $F$ ;
\item the core curve of $A^{\pm}$ bounds a disk in $\SS^3 \setminus F$ that near $\a$ lies on the positive side of $F$.
\end{enumerate}

%\label{fig:plumbing}
\begin{figure}[h]
\label{fig:plumbing}
\begin{center}
    \includegraphics[width=8cm]{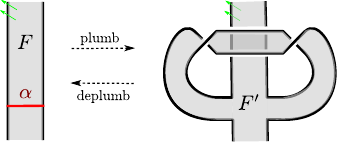}
\end{center}
\caption{The surface $F'$ is obtained from the surface $F$ by plumbing a positive Hopf band onto the arc $\alpha\subset F$. The green arrows are normal vectors indicating the orientation of $F$, these force the disk bounded by the core of the new annulus to locally be above $F$.}
\end{figure}

 We also say that $F$ is obtained by {\em deplumbing} the Hopf band $A^{\pm }$ to $F'$. See Figure \ref{fig:plumbing}. The following theorem had been conjectured by Harer in \cite{Harer:construct-fibred} and was proved by Giroux-Goodman in \cite{GG:fibred}.
\medskip

\begin{theorem}\cite{GG:fibred}
   Let $L\sb S^3$ be a fibred link with fiber surface $F$. Then $F$ can be obtained from a disk by plumbing and deplumbing Hopf bands (positive or negative). 
\end{theorem}

\begin{comment}
    
\begin{proof}
    (Delete later) The proof follows immediately from several contact geometry theorems:
    \begin{enumerate}
        \item Thurston-Wilkempenker: Any open book carries a contact structure, it is unique up to isotopy.
        \item Giroux theorem: if two open books yield isotopic contact structures, then they are related by positive stabilization (= plumbing).
        \item Eliashberg: overtwisted contact structures are classified by the homotopy of their plane fields.
        \item Torisu: negative plumbing makes contact structures overtwisted.
    \end{enumerate}
Then the proof is: suppose you have two fibered links in $S^3$, these induce contact structures in $S^3$. Now do negative plumbing to these to make them overtwisted. There are $\Z$ overtwisted contact structures in $S^3$, and after more plumbing (?), the two new open books carry isotopic contact structures. Then, by Giroux theorem, they are related by positive plumbing.
\end{proof}

And the version for knots: 

\begin{theorem}\cite{GG:fibred}
    Every fiber surface in $S^3$ whose boundary is a (fibred) knot can be obtained from a disk by plumbing and deplumbing positive trefoils and figure-eight knots.
\end{theorem}

\end{comment}

\begin{comment}
    Do we need left-handed trefoils in the statement? (GG only say positive trefoil). Are left-handed trefoils obtained from right-handed, figure eight, and deplumbings?
\end{comment}

\subsection{Plane fields and the Hopf invariant}\label{subs: plane fields and enhanced Milnor}
A fibred link induces an oriented plane field $\xi_F$ of $S^3$ as follows \cite{GG:fibred, NR:difference-index-enhanced-Milnor, Rudolph:isolated-critical}. First, the tangent planes to the fiber surfaces induce a plane field on $S^3\sm N(L)$, where $N(L)$ is a tubular neighborhood of $L$. This plane field can be extended through $N(L)$ as follows: let $(r,\theta,z)$ be oriented cylindrical coordinates around $L$. Then consider the plane field on $N(L)$ which is the kernel of $f(r)dz+r^2d\theta$ where $f:[0,1]\to [0,1]$ is smooth, $f\equiv 1$ near 0 and $f\equiv 0$ near 1. At $\p N(L)$ this coincides with the tangent planes to the fibers, so the fields glue to an oriented plane field $\xi_F$ on the whole three-sphere.
\medskip

%Under the embedding $S^3\sb \R^4$, the above oriented plane field induces a map $S^3\to \wt{G}_2(\R^2)$, that we denote by $\overline{\xi}_F$, into the Grassmannian of oriented 2-planes of $\R^4$ \cite{Rudolph:isolated-critical}. This space is homeomorphic to $S^2\t S^2$, we choose the specific homeomorphisms defined in \cite[Lemma 1.2]{Rudolph:isolated-critical}. If $l:\wt{G}_2(\R^4)\to S^2$ is the projection onto the first component of this homeomorphism, then we define $$\la(\xi_F):=-H(l\circ \overline{\xi}_F)$$where $H(f)$ is the Hopf invariant of a map $f:S^3\to S^2$ (recall that this is defined by $H(f)=lk(f^{-1}(x),f^{-1}(y))$ where $x,y$ are distinct regular values of $f$ and the fibers of $f$ have the orientation induced from $S^3$ and $S^2$). In \cite{NR:difference-index-enhanced-Milnor} this is referred as the ``enhanced Milnor number", but we will refer to it simply as the Hopf invariant of the plane field.\medskip

Now let $\xi_0$ be the oriented plane field of $S^3$ associated to the unknot with the trivial (disk) fiber. If $\xi_F$ and $\xi_0$ are in general position, then $$C_{\pm}=\{p\in S^3 \ | \ \xi_F(p)=\pm\xi_0(p)\}$$are oriented links in $S^3$, the orientation being determined by the co-orientation of the plane fields induced from the orientation of $S^3$. We denote $$\la(\xi_F):=lk(C_+,C_-),$$ this is the obstruction to homotope $\xi_F$ onto $\xi_0$. Under a trivialization of $TS^3$, $\xi_F$ corresponds to a map $f_{\xi_F}:S^3\to S^2$ whose homotopy class in $\pi_3(S^2)\cong\Z$ is determined by the {\em Hopf invariant} $H(f_{\xi_F})\in \Z$. The number $H(f_{\xi_F})-H(f_{\xi_0})$ is independent of the trivialization and coincides with $\la(\xi_F)$ \cite{NR:difference-index-enhanced-Milnor}. This is called the ``enhanced Milnor number" in \cite{NR:difference-index-enhanced-Milnor, Rudolph:isolated-critical}, but we prefer to simply call it the {\em Hopf invariant} of the plane field.

\medskip

It is shown, e.g. in \cite{GG:fibred}, that $\la(\xi_{A^+})=0$ and $\la(\xi_{A^-})=1$. Since $\la$ is additive under plumbing \cite{NR:enhanced-higher-dimensions}, it follows that if a fiber surface $F$ is obtained by plumbing $n_{\pm}$ copies of $A^{\pm}$ and deplumbing $m_{\pm}$ copies of $A^{\pm}$, then $$\la(\xi_F)=n_--m_-.$$

\section{ADO invariants}
\label{section: ADO invariants}

\def\Vla{V_{\la}}
\def\qi{q}\def\uq{\mathfrak{u}_q(\sl)}
\def\UHq{\overline{U}_q^H(\sl)}
In what follows we let $q$ be a primitive $2p$-th root of unity, say $q=e^{\pi i/p}$. For every $\la\in\C$ denote $[\la]=\frac{q^{\la}-q^{-\la}}{q-q^{-1}}$. Here $q^{a+ib}=e^{\pi (-b+ia)/p}$ for any $a,b\in\R$. For more details, see \cite{Murakami:ADO, CGP:non-semisimple}.

\def\uqp{\uq'}

\subsection{Unrolled and small quantum $\sl$}\label{subs: quantum sl2} The {\em unrolled restricted quantum group} $\UHq$ is the $\C$-algebra with generators $E,F,K^{\pm 1}, H$ and relations
\begin{align*}
[H,E]&=2E, & [H,F]&=-2F, & [H,K^{\pm 1}]&=0,\\
KE&=q^{2}EK, & KF&=q^{-2}FK, & KK^{-1}&=K^{-1}K=1, \\
[E,F]&=\frac{K-K^{-1}}{\qi-\qi^{-1}}, & E^p&=0, & F^p&=0. 
\end{align*}
This is a Hopf algebra if we define 
\begin{align*}
\De(E)&=E\ot K+1\ot E, & \e(E)&=0, & S(E)&=-EK^{-1},\\
\De(F)&=K^{-1}\ot F+F\ot 1, & \e(F)&=0, & S(F)&=-KF, \\
 \De(K)&=K\ot K, & \e(K)&=1, & S(K)&=K^{-1}, \\
 \De(H)&=H\ot 1+1\ot H, & \e(H)&=0, & S(H)&=-H.
\end{align*}
The {\em small quantum group} $\uq$ is obtained from $\UHq$ by taking the subalgebra generated by $E,F,K^{\pm 1}$ and modding out by $K^{2p}=1$. Thus, $\uq$ is finite-dimensional of dimension $2p^3$. We will consider the (ribbon) extension $\uqp$ of $\uq$ obtained by adding a group-like $k$ of order $4p$ such that $k^2=K$ and $kE=qEk$ and $kF=q^{-1}Fk$.
\medskip

\subsection{The ribbon category $\CC$}\label{subs: weight modules} A {\em weight module} over $\UHq$ is a $\UHq$-module $V$ with diagonalizable $H$-action such that $K=q^H$ as operators on $V$. The eigenvalues of $H$ are the {\em weights}. We will denote by $\CC$ the category of finite-dimensional weight $\UHq$-modules. This is braided with $c_{V,W}:V\ot W\to W\ot V$ defined by $$c_{V,W}(v\ot w)= \tau_{V,W}(\HH\TT(v\ot w))$$
where $$\TT=\sum_{m=0}^{p-1} c_m \cdot E^m\ot F^m$$
with 
\begin{align}
\label{eq: coeff cm of braiding}
    c_m=\frac{(\qa-\qa^{-1})^{m}}{[m]!}\qa^{\frac{m(m-1)}{2}},
\end{align}
$\HH$ denotes the operator $V\ot W\to V\ot W$ defined by $$\HH(v\ot w)=q^{\la \mu/2}v\ot w$$ if $v$ has weight $\la$ and $w$ has weight $\mu$, and $\tau_{V,W}$ is the transposition $v\ot w\mapsto w\ot v$. The category $\CC$ is pivotal with left evaluations and coevaluations the usual ones of vector spaces, and with right evaluations and coevaluations given by $$\rev_V(v\ot v')=v'(K^{1-p}v), \ \ \rcoev_V(1)=\sum v'_i\ot K^{p-1}v_i.$$
Here $(v_i)$ is a basis of $V$ and $(v_i')$ is the dual basis of $V^*$. It follows that there is an isomorphism $j_V:V^{**}\to V$ such that $$j_V^{-1}(v)(v')=v'(K^{1-p}v).$$

\medskip

The pivotal and braided structure are compatible, hence define a ribbon category structure on $\CC$. For any object $X$ of $\CC$, we denote by $\theta_X:X\to X$ the twist.

\subsection{Verma modules}\label{subs: Verma} For each $\la\in\C$, the Verma module $V_{\la}$ is the $\UHq$-module with basis $v_0,\dots, v_{p-1}$ and action given by 
\begin{align*}
Ev_i&=[\la-i+1]v_{i-1}, & Fv_i&=[i+1]v_{i+1}, & Hv_i&=(\la-2i)v_i, & K v_i&=q^{\la-2i}v_i,
\end{align*}
see \cite{Murakami:ADO}. The dual $V_{\la}^*$ is a $\UHq$-module with
\begin{align*}
Ev'_i&=-q^{-\la+2(i+1)}[\la-i]v'_{i+1}, & Fv'_i&=-q^{\la-2i}[i]v'_{i-1}, & Hv'_i&=(-\la+2i)v'_i, & K v'_i&=q^{-\la+2i}v'_i.
\end{align*}
Here $(v'_i)$ denote the dual basis of $V_{\la}^*$. We will refer to the above basis of $V_{\la}$ and $V_{\la}^*$ as {\em standard basis}. Any tensor product of standard basis will also be called a standard basis of the corresponding vector space (a tensor product of $V_{\la}$'s and $V_{\la}^*$'s).
\medskip

For $V_{\la}$ one has 
\begin{align}
    \label{eq: pivotal on Verma}
    j_{V_{\la}}(v''_i)=q^{(p-1)(\la-2i)}v_i
\end{align}
where $v''_i\in V_{\la}^{**}$ is the basis dual to $(v'_i)$. The twist is given by $$\theta_{V_{\la}}=q^{\la^2-(p-1)\la}\cdot\id_{V_{\la}}.$$

\def\CCZ{\CC_{[0]}}
\subsection{Relation to $\uqp$-modules} \label{subs: uqp-modules} Let $\CCZ$ be the subcategory of $\CC$ consisting of modules with integer weights. If $Y\in \CCZ$, then $K^{2p}$ acts by the identity on $Y$, so that $Y$ is a $\uqp$-module. This defines a functor $F:\CCZ\to \Rep(\uqp)$. The Hopf algebra $\uqp$ is ribbon \cite{FGST:modular} and the functor $F$ is a ribbon functor. In particular, if $Y\in \CCZ$, the twist $\theta_Y$ of $\CC$ coincides with left multiplication by the ribbon element of $\uqp$, that is $$\theta_Y=\frac{1}{4p}\sum_{i,j=0}^{4p-1}\sum_{m=0}^{p-1}c_mq^{-ij/2}k^iF^mK^{1-p}k^jE^m$$ for the $c_m\in\Q(q)$ as in (\ref{eq: coeff cm of braiding}). The inverse ribbon is given by 
$$\theta_Y^{-1}=\frac{1}{4p}\sum_{i,j=0}^{4p-1}\sum_{m=0}^{p-1}c_mq^{-ij/2}S(k^jE^m)K^{1-p}  k^iF^m.$$
This applies in particular to the $\UHq$ module $Y=V_{\la}^*\ot V_{\la}$ (for any $\la \in\C$) which clearly has integer weights.

\begin{comment}
\subsection{The module $V_{\la}^*\ot V_{\la}$} Let $Y=V_{\la}^*\ot V_{\la}$. This has only integer weights, hence it is a $\uqp$-module (where the action of $k$ is defined simply as $q^{H/2}$). Since $$q^{H\ot H/2}=\sum_{i,j=0}^{4p-1}q^{-ij/2}k^i\ot k^j$$
on $\uqp$-modules, the braiding and twist of $\CC$ on tensors products of $Y$ can be entirely computed in terms of the universal $R$-matrix $R\in \uqp\ot\uqp$. 
\end{comment}

\subsection{ADO invariants}
\label{subs: ADO invariants}
We will assume the reader is familiar with the Reshetikhin-Turaev invariant $F(T)$ of framed, oriented, $\CC$-colored tangles $T\sb \R^2\t [0,1]$, see \cite{Turaev:BOOK1}. If all components of $T$ have the same color $V\in\CC$, we will denote the invariant by $F(T,V)$. If $L$ is a (framed, oriented) link in $S^3$ and $L_o$ is a $(1,1)$-tangle (oriented upwards at the endpoints) obtained by opening $L$ at a strand, then $F(L_o, V_{\la})$ is a $\UHq$-morphism $\Vla\to \Vla$. Since $\Vla$ is simple (for generic $\la$), it must be $F(L_o,V_{\la})=\a_{\la} \cdot\id_{\Vla}$ for some $\a_{\la}\in \C$. This turns out to be independent of the place where the cut is happening, hence it is a link invariant that we will denote by $\ADO_p'(L,\la)$. The ADO invariant of the unframed, oriented link $L$ is $$\ADO_p(L,\la):=q^{(p-1)\la \cdot w(L)}\ADO_p'(L,\la)$$
where $w(L)$ is the writhe of the (framed) link $L$. Setting $t=q^{2\la}$ this becomes a Laurent polynomial in $t$ (up to an overall fractional power of $t$) which we denote $\ADO_p(L,t)$. More precisely $$\ADO_p(L,t)\in t^{\frac{(p-1)(s-1)}{2}}\Z[q^2][t^{\pm 1}]$$
where $s$ is the number of components of $L$ \cite{LNV:genus-unrolled}. Here $\Z[q]\sb \C$ is the ring of cyclotomic integers. This is not symmetric in $t$ (i.e. invariant under $t\mapsto t^{-1}$), but after the change of variable $t=q^{-2}x$ it becomes symmetric in $x$, see \cite{MW:unified}.
\medskip

Let $H_{+}$ and $H_-$ be the positive and negative Hopf links respectively. Their ADO invariants are given by (see \cite{Murakami:ADO}):

\begin{equation}
    \begin{aligned}
\label{eq: ADO of Hopf link}
    \ADO_p(H_+,t)&=\sum_{i=0}^{p-1}t^{\frac{p-1-2i}{2}}q^{-2i}=q^2t^{-(p-1)/2}+\dots+t^{(p-1)/2} \\
   \ADO_p(H_-,t)&=\sum_{i=0}^{p-1}t^{\frac{p-1-2i}{2}}q^{-2(1+i)}=t^{-(p-1)/2}+\dots+ q^{-2}t^{(p-1)/2}.
\end{aligned}
\end{equation}

In $x=q^{2}t$;

\begin{equation}
\label{eq: ADO of Hopf link}
    \begin{aligned}
    \ADO_p(H_+,x)&=-qx^{-(p-1)/2}+\dots-qx^{(p-1)/2} \\
    \ADO_p(H_-,x)&=-q^{-1}x^{-(p-1)/2}+\dots-q^{-1}x^{(p-1)/2}.
\end{aligned}
\end{equation}

\section{Proof of main theorems}
\label{section: proof of main thm}

In this section we prove Theorem \ref{Thm: main fibred}. As mentioned in the introduction, this follows from Theorem \ref{Theorem: top coefficient of ADO after plumbing} and the Giroux-Goodman theorem. Theorem \ref{Theorem: top coefficient of ADO after plumbing} follows from Proposition \ref{prop: ADO after plumbing} which studies the effect of a single Hopf plumbing on ADO invariants. In turn, Proposition \ref{prop: ADO after plumbing} follows from our study of the top term of the ribbon element in Subsection \ref{subs: top term of ribbon} and the ideas of the proof of the genus bound in \cite{LNV:genus-unrolled}, which we recall in Subsection \ref{subs: bounds universal}. We carry our proofs using the relation to universal $\uqp$-invariants as explained in Subsection \ref{subs: uqp-modules}. Our main theorems are proved in Subsection \ref{subs: PROOF OF MAIN}. In Subsection \ref{subs: corollaries} we deduce some corollaries for specific classes of fibred links.

\def\SS{S}

\subsection{The top term of the ribbon element}
\label{subs: top term of ribbon}
We set up to understand the effect of Hopf plumbing on ADO invariants. This is done in Proposition \ref{prop: top degree part of twist (or double braiding)} below. We need a couple lemmas for this.

\begin{lemma}
\label{lemma: E,F power p-1 on i different j}
    We have
    \begin{enumerate}
        \item $E^{p-1}(v_i'\ot v_j)=0$ if $j<i$.
        \item $F^{p-1}(v_i'\ot v_j)=0$ if $j>i$.
    \end{enumerate}
\end{lemma}

  \begin{proof}
        Each standard basis vector $v'_k\ot v_l$ has weight $w=2(k-l)$. 
        %Since $$E(v'_k\ot v_l)=c(v'_{k+1}\ot v_l)+d(v'_k\ot v_{l-1})$$for some $c,d\in \Q(q,q^{\la})$, it follows that $E$ increases weight by two. 
        Since $E$ increases weight by 2, $E^{p-1}(v'_i\ot v_j)$ has weight $2(i-j+p-1)>2(p-1)$ as $j<i$. Since there are no vectors with such weight, this must be zero. A similar argument applies to $F^{p-1}$.
    \end{proof}

\begin{lemma}
\label{lemma: E power p-1 on i=j}
    We have 
    $E^{p-1}(v'_i\ot v_i)=\displaystyle \left(\prod_{s=0}^{p-2}[\la-s]\right) v'_{p-1}\ot v_0.$

\end{lemma}
  \begin{proof}
      For any $l,k$ we have 
      \begin{equation}
          \label{eq: E on v}
          E^lv_i=[\la-i+1]\dots [\la-i+l]v_{i-l}
      \end{equation}
    
      and 
       \begin{equation}
          \label{eq: E on v'}
          E^kv'_i=(-1)^kq^{-k\la+2ki+k(k+1)}[\la-i]\dots [\la-(i+k-1)]v'_{i+k}.
      \end{equation}
     By the $q$-binomial theorem $$E^{p-1}(v'_i\ot v_i)=\sum_l\binom{p-1}{l}_{q^{-2}}E^lv'_i\ot K^lE^{p-1-l}v_i.$$
     The $l$-th term on the right hand side is a scalar multiple of $v'_{i+l}\ot v_{i+l-(p-1)}$. This is nonzero only for $i+l=p-1$ hence 
$$E^{p-1}(v'_i\ot v_i)=\binom{p-1}{p-1-i}_{q^{-2}}E^{p-1-i}v'_i\ot K^{p-1-i}E^{i}v_i.$$
     Now, using (\ref{eq: E on v}) and (\ref{eq: E on v'}) with $l=i$ and $k=p-1-i$, and that $\binom{p-1}{k}_{q^{-2}}=(-1)^kq^{k(k+1)}$ we get
     $$
         E^{p-1}(v'_i\ot v_i)=(-1)^{2k}q^{k(k+1)-k\la+2ki+k(k+1)+k\la}\left(\prod_{s=0}^{p-2}[\la-s]\right)v'_{p-1}\ot v_0$$
    where $k=p-1-i$. It is easy to see that the power of $q$ in the middle is exactly $q^{2kp}$. Using that $q$ is a $2p$-th root of unity, the lemma follows.
  \end{proof}

   \begin{lemma}
   \label{lemma: F power p-1 on i=j}
   We have
        $F^{p-1}(v'_{p-1}\ot v_0)=[p-1]!\cdot \rcoev_V$.
   \end{lemma}
\begin{proof}
    For any $k,l$ we have
    \begin{equation}
        \label{eq: F on v}
        F^lv_0=[l]!v_{l}
    \end{equation}
and
    \begin{equation}
        \label{eq: F on v'}
        F^kv'_{p-1}=(-1)^kq^{k\la+k(k+1)}[p-1]\dots [p-k]v'_{p-1-k}.
    \end{equation}

    As before, using these two formulas and the $q$-binomial theorem we find (letting $k=p-1-l$):
\begin{align*}
    F^{p-1}(v'_{p-1}\ot v_0)&=\sum_l (-1)^lq^{l(l+1)}K^{-l}F^{p-1-l}v'_{p-1}\ot F^lv_0\\
    &=\sum_l(-1)^{l+k}q^{l(l+1)+(\la-2l)l+k\la+k(k+1)}[p-1]!v'_{l}\ot v_{l}\\
    &=[p-1]!\sum_l(-1)^{l+k}q^{(k+l)(k-l+1)+(k+l)\la}v'_{l}\ot v_{l}\\
     &=[p-1]!\sum_l(-1)^{p-1}q^{(p-1)(p-2l)+(p-1)\la}v'_{l}\ot v_{l}\\
     &=[p-1]!\sum_l(-1)^{p-1}q^{(p-1)p+2l+(p-1)\la}v'_{l}\ot v_{l}\\
      &=[p-1]!\sum_lq^{2l+(p-1)\la}v'_{l}\ot v_{l} \ \text{ since $q^p=-1$}\\
      &=[p-1]!\sum_lv'_{l}\ot K^{p-1}v_{l}=[p-1]!\rcoev_V.
\end{align*}
\end{proof}

\begin{lemma}\cite{LNV:genus-unrolled}
\label{lemma: degree bound on X}
      Let $z\in\uq$ and consider the module $X=V_{\la}\ot V_{\la}^*$ with the standard basis. Then each coefficient of the matrix representing the action of $z$ on $V_{\la}\ot V_{\la}^*$ is a polynomial in $\Q(q)[q^{-2\la}]$ of degree $\leq p-1$ (in $q^{-2\la}$).

\end{lemma}
\medskip

Similarly, we have:

\begin{lemma}
\label{lemma: degree bound on Y}
  The products $F^lE^l\in \uq$ and $E^lF^l\in \uq$ act on $Y=V_{\la}^*\ot V_{\la}$ (in the standard basis) with coefficients polynomials in $q^{2\la}$ of degree $\leq l$.
\end{lemma}
\begin{proof}
    This follows easily from the formulas of Subsection \ref{subs: Verma}.
\end{proof}

Let $\theta_Y$ be the twist of $Y$. Since $Y$ is a $\uq$-module, this is simply the action of the ribbon element of $\uqp$.

\medskip
\begin{proposition}
\label{prop: top degree part of twist (or double braiding)}
Let $Y=V_{\la}^*\ot V_{\la}$. Then $$\theta_Y=\ADO_p(H_-)\cdot (\rcoev_V\circ\lev_V) +B$$
and $$\theta_Y^{-1}=\ADO_p(H_+)\cdot(\rcoev_V\circ\lev_V)+B'$$
where $B$ and $B'$ are morphisms $Y\to Y$ whose coefficients on the standard basis of $Y$ are (plain, non-Laurent) polynomials in $q^{2\la}$ of degree $<p-1$. 
\end{proposition}

\begin{comment}
 
\begin{proposition}
\label{prop: top degree part of twist (or double braiding)}
Let $Y=V_{\la}^*\ot V_{\la}$. Then $$\theta_Y=p(q^{\la})\cdot (\rcoev_V\circ\lev_V) +B$$
where $p(q^{\la})\in\Z[q][q^{\la}]$ is a Laurent polynomial of the form $$p(q^{\la})=q^{-(p-1)\la}+\text{ terms of lower degree }+q^{-2}q^{(p-1)\la}$$
and where
$B:Y\to Y$ is a morphism whose coefficients on the standard basis of $Y$ are (plain, non-Laurent) polynomials in $q^{2\la}$ of degree $<p-1$. Similarly 
$$\theta_Y^{-1}=\ADO_p(H_+)\rcoev_V\circ\lev_V+B'$$
where $B':Y\to Y$ is a morphism with the same property as $B$.
\end{proposition}

\end{comment}

\begin{proof}

\def\ttop{\theta_{top}}\def\KK{\mathcal{K}}
   
   We can write $$\theta_Y=\frac{1}{4p}\sum_{i,j=0}^{4p-1}\sum_{m=0}^{p-1}c_mq^{-ij/2}k^iF^mK^{1-p}k^jE^m$$ where $c_m\in\Q(q)$ is the coefficient of (\ref{eq: coeff cm of braiding}). Let $\ttop$ be the ``top" term of $\theta_Y$, that is
     $$\ttop=\frac{c_{p-1}}{4p}\sum_{i,j=0}^{4p-1}q^{-ij/2}k^iF^{p-1}K^{1-p}k^jE^{p-1}=\KK\cdot c_{p-1}F^{p-1}K^{1-p}E^{p-1}$$
   where $\KK=\frac{1}{4p}\sum_{i,j=0}^{4p-1}q^{-ij/2+j(p-1)}k^{i+j}$. A simple computation shows that $\KK$ acts on a vector of weight $\kappa$ by $(-1)^{\kappa}q^{-\kappa+\kappa^2/2}$, in particular, it is the identity on weight zero vectors. When $m<p-1$, Lemma \ref{lemma: degree bound on Y} implies that the $m$-th term in the sum for $\theta_Y$ only contributes coefficients of degree $<p-1$. Thus, it suffices to prove that $$\ttop=\ADO_-(H_-)\cdot (\rcoev_V\circ\lev_V) +U$$
where $U$ is an operator with the same property as $B$. We will actually show that 
$$\ttop=p(q^{\la})\cdot (\rcoev_V\circ\lev_V) +U$$
where $p(q^{\la})$ is a polynomial in $q^{\la}$ with the same top/least powers as $\ADO_p(H_-)$ and with the same coefficients on those powers, this is clearly equivalent to our proposition. We prove this by evaluating on $v'_i\ot v_j$ and distinguishing two cases:
   \medskip
   
   \textbf{Case 1, $i\neq j$:} since $\lev_V(v'_i\ot v_j)=0$, we must show that $\ttop(v'_i\ot v_j)$ has coefficients of degree $<p-1$. If $j<i$ this is zero by part $a)$ of Lemma \ref{lemma: E,F power p-1 on i different j} and we are done. Now, if $j>i$ we first commute $$F^{p-1}E^{p-1}=E^{p-1}F^{p-1}+ \sum_{a,b,c} d_{a,b,c} E^aF^bK^c$$
   where $a<p-1$ and $d_{a,b,c}$ are coefficients. Using part $b)$ of Lemma \ref{lemma: E,F power p-1 on i different j} we get $$\ttop(v'_i\ot v_j)=c_{p-1}\KK q^{-(p-1)^2}K^{1-p}F^{p-1}E^{p-1}(v'_i\ot v_j)= \sum_{a,b,c} d'_{a,b,c} E^aF^bK^c(v'_i\ot v_j)$$
   for some new coefficients $d'_{a,b,c}$, hence this case follows from Lemma \ref{lemma: degree bound on Y} since $a<p-1$.

   \medskip

   \textbf{Case 2, $i=j$:} Here we can drop the $\KK$ factor in $\ttop$ since $\KK$ is the identity on vectors of weight zero. Combining Lemmas \ref{lemma: E power p-1 on i=j} and \ref{lemma: F power p-1 on i=j} we find 
   \begin{align*}
       \ttop(v'_i\ot v_i)&=c_{p-1}F^{p-1}\left( q^{-2}\displaystyle \prod_{s=0}^{p-2}[\la-s]\cdot v'_{p-1}\ot v_{0}\right)\\
       &=c_{p-1}[p-1]!q^{-2}\left(\prod_{s=0}^{p-2}[\la-s]\right)\cdot \rcoev_V\\
       &=(q-q^{-1})^{p-1}q^{(p-1)(p-2)/2}q^{-2}\left(\prod_{s=0}^{p-2}[\la-s]\right)\cdot (\rcoev_V\circ\lev_V)(v'_i\ot v_i)\\
       &=p(q^{\la})\cdot (\rcoev_V\circ\lev_V)(v'_i\ot v_i)
   \end{align*}
for some Laurent polynomial $p(q^{\la})\in \Z[q^{\pm 1}][q^{\la}]$. The product $\prod_{s=0}^{p-2}[\la-s]$ is a Laurent polynomial in $q^{\la}$ of the form $$\frac{1}{(q-q^{-1})^{p-1}}\left((-1)^{p-1}q^{-(p-1)\la+k}+\text{ terms of lower degree }+ q^{(p-1)\la-k}\right)$$
   where $k=\sum_{s=0}^{p-2} s=(p-1)(p-2)/2$. Hence $p(q^{\la})$ has the form
   \begin{align*}
       p(q^{\la})&=q^{-2}((-1)^{p-1}q^{-(p-1)\la+(p-1)(p-2)}+\dots+q^{(p-1)\la})\\
       &=(-1)^{p-1}q^{-(p-1)\la+p^2-3p}+\dots+q^{-2}q^{(p-1)\la}\\
       &=q^{-(p-1)\la}+\dots+q^{-2}q^{(p-1)\la}.
   \end{align*}
   Here we used that $q^{p^2-3p}=(-1)^{p-1}$ since $q^p=-1$. This shows that the top/least powers and their coefficients are the same for $p(q^{\la})$ and $\ADO_p(H_-)$ (see Subsection \ref{subs: ADO invariants}), hence proves the first part of the proposition. The second assertion follows from a similar computation using the formula for the inverse ribbon element given in Subsection \ref{subs: uqp-modules}. %that $R^{-1}=(S\ot \id)(R)$ for the $R$-matrix $R$ of $\uqp$.
\end{proof}

\def\top{\text{top}_{q^{\la}}}
\def\topco{\text{top-coeff}}

\subsection{Bounds from universal invariants} 
\label{subs: bounds universal}

We will need the essential part of the argument of the genus bound for ADO as proved in \cite{LNV:genus-unrolled}. Let $T$ be a framed, oriented tangle in $\R^2\t [0,1]$ with $l$ components $T_1,\dots, T_l$ which are all open and with endpoints in $\R^2\t \{1\}$. We suppose all components are $\CC$-colored with $X=V_{\la}\ot V_{\la}^*$. To simplify notation in the formulas for $F_{\CC}(T,X)$ below we will suppose that, when $T$ is drawn in the plane, $\p T$ consists of $2l$ points $x_1,\dots, x_{2l}$ ordered from left to right and for each $1\leq i\leq l$, $\p T_i=\{x_i, x_{l+i}\}$ and $T_i$ is oriented from $x_i$ to $x_{l+i}$, see Figure \ref{fig.TangleT} (left). 

\begin{figure}[htp!]
\includegraphics[width=10cm]{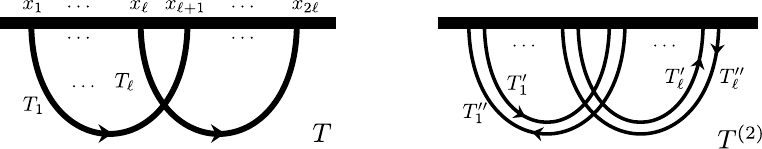}
\caption{The tangle $T$ and its doubled version $T^{(2)}$.}
\label{fig.TangleT}
\end{figure}

Thus, the Reshetikhin-Turaev invariant $F_{\CC}(T,X)$ is a morphism $$F_{\CC}(T,X):\C\to (X^*)^{\ot l}\ot X^{\ot l}.$$ 

\def\coev{\text{coev}}\def\Mod{\text{-mod}}
As explained in Subsection \ref{subs: uqp-modules}, $X$ is a $\uqp$-module and the braiding and twist coming from $\CC$ coincide with the braiding and twist defined from the ribbon structure of $\uqp$, hence $F_{\CC}(T,X)=F_{\uqp\Mod}(T,X).$
Now, since $\uqp$ is a finite-dimensional ribbon Hopf algebra, $F_{\uqp\Mod}$ is determined by its {\em universal invariant} \cite{Habiro:bottom-tangles}. This is an element $Z_T\in \uqp^{\ot l}$ satisfying $$F_{\uqp\Mod}(T,X)=(\id_{(X^*)^{\ot l}}\ot L_{Z_T})\circ\coev_{X^{\ot l}}.$$
Here $\coev$ denotes the standard right coevaluation of vector spaces and $L_{Z_T}$ denotes left multiplication by $Z_T$ on $X^{\ot l}$. More explicitly, this means that
$$F_{\CC}(T,X)(1)=\sum_{i_1,\dots,i_l\in I} w'_{i_1}\ot\dots \ot w'_{i_l}\ot (a_{1} w_{i_1})\ot\dots\ot (a_{l} w_{i_l}) $$
where $Z_T=\sum a_{1}\ot\dots\ot a_{l}$, $\{w_i\}_{i\in I}$ is a basis of $X$ and $\{w'_i\}_{i\in I}$ is the dual basis. If we take the standard basis $\{v_i\ot v'_j\}$ of $X$ and apply Lemma \ref{lemma: degree bound on X} to the last $l$ tensor slots of $F_{\CC}(T,X)$ in the above formula, we get:

\begin{lemma}
    The vector $F_{\CC}(T,X)(1)\in (X^*)^{\ot l}\ot X^{\ot l} $
    can be written as a linear combination of standard basis vectors where each coefficient is a polynomial in $q^{-2\la}$ of degree $\leq l(p-1)$.
\end{lemma}

Now let $T^{(2)}$ be the framed, oriented $\CC$-colored tangle obtained from $T$ as follows: each component $T_i$ is duplicated (with its framing) to get a pair $T'_i, T''_i$. Near the endpoint $x_{l+i}$ of $T_i$ we suppose $T'_i$ is to the left of $T''_i$ and that $T'_i$ is oriented upwards while $T''_i$ is oriented downwards. It follows that near the starting point $x_i$ of $T_i$, $T'_i$ is to the right of $x_i$ oriented downwards and $T''_i$ is to the left oriented upwards. See Figure \ref{fig.TangleT} (right). If all components of $T^{(2)}$ are colored with $V_{\la}$, it follows that the Reshetikhin-Turaev invariant is a morphism $$F_{\CC}(T^{(2)}, V_{\la}):\C\to (V_{\la}\ot V_{\la}^*)^{\ot 2l}=X^{\ot 2l}.$$

It is easy to see that $$F_{\CC}(T^{(2)}, V_{\la})=(J^{\ot l}\ot \id_{X^{\ot l}})\circ F_{\CC}(T,X)$$
where $J:X^*\to X$ is the isomorphism coming from the pivotal structure. Since each $J$ brings an overall power $q^{\la}(p-1)$ the previous lemma implies:

\begin{lemma}
\label{lemma: bound for degree on RT of bottom tangle}
    Write $F_{\CC}(T^{(2)},X)(1)\in X^{\ot 2l}$
    as a linear combination of standard basis vectors. Then each coefficient is a Laurent polynomial in $q^{\la}$ of maximal degree $\leq l(p-1)$ and minimal degree $\geq -l(p-1)$.
\end{lemma}

In general, if the endpoints of $T$ are displayed in a different way along $\R^2\t\{1\}$, the formulas for $F_{\CC}(T,X)$ above differ by a permutation isomorphism (of vector spaces) and the lemmas remain unchanged (except for the target vector space of $F(T,X)$).

\def\top{\text{max-deg}_{q^{\la}}}

\subsection{Proof of main theorems} 
\label{subs: PROOF OF MAIN}
In what follows, if $h(q^{\la})$ is a Laurent polynomial in $q^{\la}$ we denote by $\top h(q^{\la})$ the maximal power of $q^{\la}$ appearing in $h(q^{\la})$.  %We similarly use the notation $\topx p(x)$ for a Laurent polynomial in the variable $x$.

\medskip

\begin{proposition}
\label{prop: ADO after plumbing}
Let $F$ be a connected Seifert surface for a link $L$ and let $F'$ be obtained from $F$ by plumbing a Hopf band $A^{\pm}$. Let $L'=\p F'$ and $l=\dim H_1(F)$. Then 
    \begin{equation*}
    \label{eq: ADO after plumbing}
        \ADO_p(L',\la)=\ADO_p(H_{\pm},\la)\ADO_p(L,\la)+f(q^{\la})
    \end{equation*}
    where $f(q^{\la})$ is a polynomial in $q^{\la}$ with $\top f(q^{\la})<(l+1)(p-1)$. 
\end{proposition}
\begin{proof}
      We will prove it for $A_-$-plumbing, the positive case is similar. Let $\a\sb F$ be the arc where the plumbing occurs. Since $F$ is connected, after an isotopy, we can suppose $F$ is obtained from a $l$-component (framed) tangle $T$ with ends on $\R^2\t\{1\}$ by attaching a disk on top, and that the arc $\a$ is a vertical arc on that disk: 
      \begin{figure}[H]
          \centering
          \includegraphics[width=4cm]{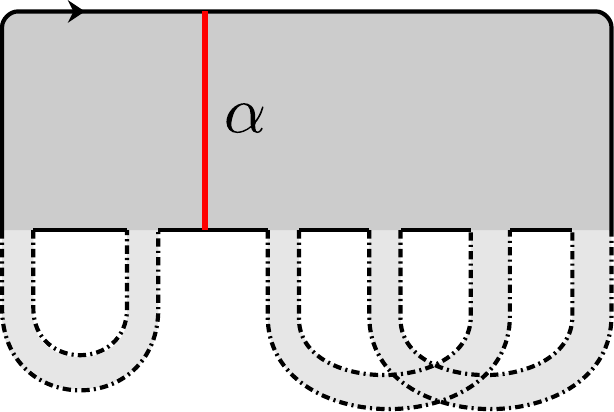}
    
      \end{figure}
      Do the plumbing along $\a$ and apply Proposition \ref{prop: top degree part of twist (or double braiding)} to the Hopf band that was plumbed:
    \medskip

%\begin{figure}[H]  \centering \includegraphics[width=10cm]{proof.png}\end{figure}

\begin{figure}[H]
\begin{align*}
\ADO_p(L')=\ADO_p&\left(\rule{0cm}{2cm}\ \begin{matrix}
\includegraphics[width=4cm]{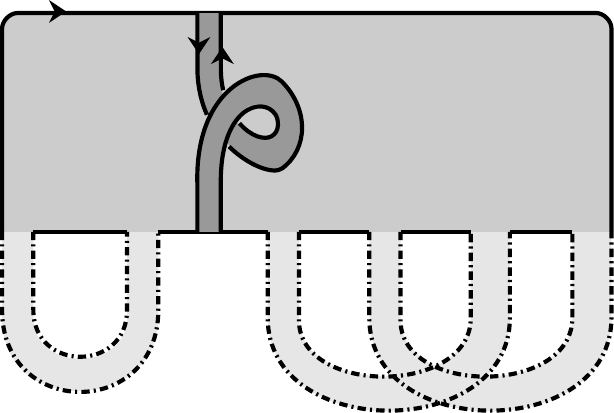}
\end{matrix}\ \right) =\\
\ADO_p(H_-)\cdot\ADO_p\ 
&\left(\rule{0cm}{2cm}\begin{matrix}
\includegraphics[width=4cm]{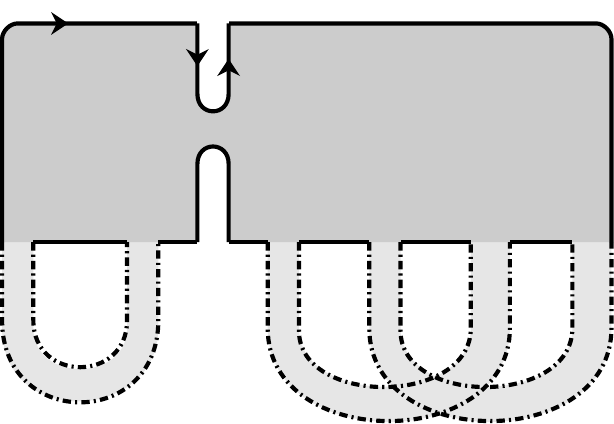}
\end{matrix}\ \right)  +
\ADO_p\left(\rule{0cm}{2cm}\ 
\begin{matrix}
\includegraphics[width=4cm]{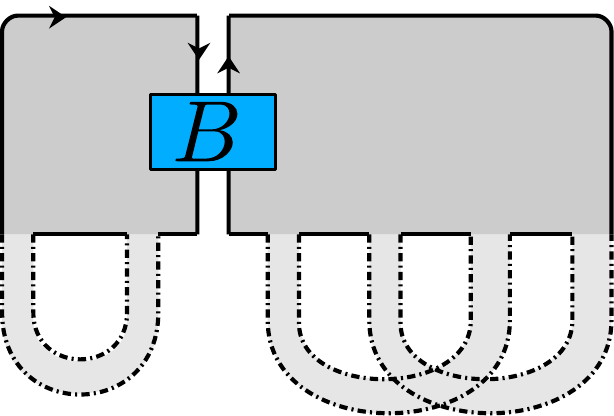}
\end{matrix}\ \right)
\end{align*}
\end{figure}

\medskip

In the second line of the figure above, the LHS is precisely $\ADO_p(H_-)\ADO_p(L)$. Note that we defined $\ADO_p$ only for links, but our definition readily extends to graphs with coupons as the graph $G$ in the RHS of the above equation: if $G_o$ is the graph $G$ opened at a point, say at the top right corner, then $F_{\CC}(G_o)$ is a morphism $V_{\la}\to V_{\la}$ that must be multiplication by a scalar $f(q^{\la})$. We define this scalar as $\ADO_p(G)$. We now show that $f(q^{\la})$ satisfies the bound of the proposition. Let $T$ be the tangle formed by the bands of the surface. Then $$f(q^{\la})\id_{V_{\la}}=F_{\CC}(G_o)=(\rev_{V_{\la}}\ot \id_{V_{\la}})\circ(\id_{V_{\la}}\ot\lev_{V_{\la}}^{\ot k-1}\ot B\ot \lev_{V_{\la}}^{\ot 2l-k})\circ (F_{\CC}(T^{(2)},V_{\la})\ot \id_{V_{\la}}) $$
so that
$$f(q^{\la})v_0=(\rev_{V_{\la}}\ot \id_{V_{\la}})\circ(\id_{V_{\la}}\ot\lev_{V_{\la}}^{\ot k-1}\ot B\ot \lev_{V_{\la}}^{\ot 2l-k})[F_{\CC}(T^{(2)},V_{\la})(1)\ot v_0] $$
where $k$ is the number of feet of 1-handles attached to the left of the arc $\a$ on the bottom of the disk and $v_0$ is the highest weight vector of $V_{\la}$. By Lemma \ref{lemma: bound for degree on RT of bottom tangle}, $F(T^{(2)}, V_{\la})(1)$ has coefficients polynomials in $q^{\la}$ with maximal power $\leq l(p-1)$ in the standard basis. By Proposition \ref{prop: top degree part of twist (or double braiding)}, the morphism $B$ brings powers of $q^{\la}$ up to $2(p-2)$, note also that the left evaluations $\lev_{V_{\la}}$ bring no powers of $q^{\la}$ into the formula. Finally, the right evaluation $\rev_{V_{\la}}$ multiplies everything by an overall $q^{-(p-1)\la}$. In total, it follows that the maximal power of $q^{\la}$ in the coefficients of the final expression (which is $f(q^{\la})v_0$) is at most $$q^{l(p-1)\la-(p-1)\la+2\la(p-2)}=q^{(l+1)(p-1)\la-2\la}.$$
This proves the desired degree bound on $f(q^{\la})$, hence the theorem.
\end{proof}

\def\ttop{\text{top coefficient}_{q^{\la}}}

\def\topx{\text{top coefficient}_x}

\begin{theorem}
\label{Theorem: top coefficient of ADO after plumbing}
    Let $L_0\sb S^3$ be a link and let $p\geq 2$ be such that $\ADO_p(L_0)$ has maximal degree allowed by (\ref{eq: genus bound for ADO}). Suppose $L$ is obtained from $L_0$ by plumbing $n_+$ (resp. $n_-$) positive (resp. negative) Hopf links and deplumbing $m_+$ (resp. $m_-$) positive (resp. negative) Hopf links. Then $\ADO_p(L)$ also has maximal degree allowed by (\ref{eq: genus bound for ADO}) and
\begin{align*}
%\label{eq: top power of x}
 \text{top coefficient of $\ADO_p(L,x)$}=(-q)^{n_+-n_--m_++m_-}\cdot \text{top coefficient of $\ADO_p(L_0,x)$}.
\end{align*}
\end{theorem}
\begin{proof}
Let $L_1\sb S^3$ be any $s_1$-component link such that $\ADO_p(L_1)$ has maximal degree and suppose $L_2$ is obtained by plumbing a single Hopf band $A^{\pm}$ to a minimal genus Seifert surface of $L_1$. Let $l=2g(L_1)+s_1-1$. The genus bound (\ref{eq: genus bound for ADO}) of ADO implies $$\top \ADO_p(L_2,\la)\leq (l+1)(p-1), \ \top \ADO_p(L_1,\la)\leq l(p-1).$$
    Since $\top \ADO_p(H_{\pm})=p-1$ (see (\ref{eq: ADO of Hopf link}), recall $x=q^{2\la+2}$), the condition on $f(q^{\la})$ of Proposition \ref{prop: ADO after plumbing} implies that 
    \begin{align}
         \label{eq: max deg after plumbing iff max deg before}
    \top\ADO_p(L_2,\la)=(l+1)(p-1) \ \text{ if and only if } \ \top\ADO_p(L_1,\la)=l(p-1)
        \end{align}
    and in such a case, one has
    \begin{align*}
        \ttop \ADO_p(L_2,\la)=\ttop\ADO_p(H_{\pm})\cdot\ttop \ADO_p(L_1,\la).
    \end{align*}
    In the variable $x=q^{2\la+2}$ this reads 
    \begin{align}
     \label{eq: top coeff of ADO under plumbing}
        \topx \ADO_p(L_2,x)=(-q)^{\e}\topx \ADO_p(L_1,x)
    \end{align}
    where $\e=\pm 1$ is the sign of the plumbing. Now, starting with a fixed link $L$ with $\ADO_p(L)$ of maximal degree, (\eqref{eq: max deg after plumbing iff max deg before}) implies that any link $L'$ obtained by plumbing/deplumbing on $L$ will have the same property on the top degree (and hence on the degree by symmetry under $x\mapsto x^{-1}$). Moreover, since $q$ is invertible, (\ref{eq: top coeff of ADO under plumbing}) implies the desired formula for the top coefficient of $x$ in $\ADO_p(L')$.
\end{proof}

\begin{proof}[Proof of Theorem \ref{Thm: main fibred}]
If $L$ is fibred then, by the Giroux-Goodman theorem, we can let $L_0$ be the unknot in Theorem \ref{Theorem: top coefficient of ADO after plumbing} and suppose $L$ is obtained by plumbing/deplumbing Hopf links on $L_0$ as stated there. Since $N_p(L_0)=1$, we get that $N_p(L)$ has maximal degree and that its top coefficient is $(-q)^{n_+-n_--m_++m_-}$. The explicit form for the coefficient in terms of the Hopf invariant follows from $$2g(L)+s-1=\dim H_1(F;\R)=n_++n_--m_+-m_- \ \text{ and } \ \la(\xi_F)=n_--m_-,$$
   see Subsection \ref{subs: plane fields and enhanced Milnor}.
\end{proof}

\begin{corollary}
\label{corollary: all ADOs detect enhanced Milnor}
    The family $\{\ADO_p(L)\}_{p\geq 2}$ of ADO invariants of a fibred link in $S^3$ detects the Hopf invariant of the plane field $\xi_F$.
\end{corollary}
 
\begin{proof}
From the formula for the top coefficient of $\ADO_p(L)$ in Theorem \ref{Thm: main fibred}, it follows that $\ADO_p$ detects $2g(L)+s-1-2\la(\xi_F)$ mod $2p$. But it also detects $2g(L)+s-1$ as the top degree of the $x$-variable, hence it detects $2\la(\xi_F)$ mod $2p$. Since this is true for all $p\geq 2$, the corollary follows.
\end{proof}

\begin{remark}
    In Theorem \ref{Theorem: top coefficient of ADO after plumbing} it is necessary that $N_p(L_0,x)$ has maximal degree. For instance, the Conway knot $K_C$ can be obtained by plumbing 3 Hopf bands to a 4-component link $L_0$ (see the top of Figure 5 in \cite{Gabai:foliations-genera}) with $N_4(L_0,x)=0$, but $N_4(K_C,x)\neq 0$. The link $L_0$ is obtained from an unoriented Hopf link in which both components are doubled, and one of the parallel pairs is given a parallel orientation while the other has opposite orientation. We checked in Mathematica that $N_4(L_0)=0$, but it can be proved that $N_p(L_0)=0$ for all $p$ using computations from \cite{CGP:remarks-unrolled} (namely Theorem 5.2 and Lemma 6.6).
    %\textcolor{blue}{Can we make a list of non-fibred links with maximal degree ADO? Alternating, some iterated twisted annuli, what else?}
\end{remark}

\subsection{Special classes of fibred links}
\label{subs: corollaries}
A link in $S^3$ is {\em strongly quasi-positive} (SQP) if it is the closure of a braid that is a product of elements of the form $$(\s_j\s_{j+1}\dots\s_{i-1})\s_i(\s_j\s_{j+1}\dots\s_{i-1})^{-1}.$$ Equivalently (and more geometrically), a link is SQP if it is the boundary of a quasi-positive surface, that is, a subsurface $F$ of the fiber $F_T$ of a positive torus link with the property that $\pi_1(F)\to \pi_1(F_0)$ is injective \cite{Rudolph:quasipositive3-characterization}.

%\textcolor{blue}{Can we express $2g+s-1$ directly from the braid in an easy way? Here the fiber surface is the braid surface indeed.}

%\textcolor{blue}{Another class of fibered links: homogeneous \cite{Stallings:constructions-fibred}. Baader's theorem: positive if and only if SQP and homogeneous. Thus homogeneous non-positive must admit figure-eight fibers.}

\begin{corollary}
\label{corollary: positive braid and fibred SQ}
    Suppose $L$ is fibred and strongly quasi-positive. Then $N_p(L)$ has maximal degree and its top coefficient is $(-q)^{2g(L)+s-1}$. 
\end{corollary}
\begin{proof}
     Fibred strongly-quasipositive links are precisely the links that can be obtained by plumbing/deplumbing positive Hopf bands: this is because a plumbing of two surfaces is strongly quasi-positive if and only if each of the two surfaces is so \cite{Rudolph:quasipositive-plumbing}, and the positive Hopf band is quasi-positive while the negative Hopf band is not \cite{Rudolph:quasipositive3-characterization}. Hence, $n_-=m_-=0$ for a fibred SPQ link so that $\la(F)=0$ and the corollary follows from Theorem \ref{Thm: main fibred}.
\end{proof}

\begin{remark}
    It is interesting to note that, for fibred links, strong quasi-positivity is equivalent to say that the induced contact structure on $S^3$ (by the open book $(F,L)$) is the tight one \cite{Hedden:notions-positivity}. Also, deplumbing is necessary in the class of SQP knots (see e.g. \cite{Etnyre-Ozbagci:invariants-contact-from-open-books}).
\end{remark}

\medskip

Another class of links that are fibered are homogeneous braid closures, as shown in \cite{Stallings:constructions-fibred}. A braid $\b\in B_n$ is said to be {\em homogeneous} if it can be written as a product of the standard generators $\s_1,\dots,\s_{n-1}$ in a such a way that every generator appears and all appearances of a generator $\s_i$ have the same sign (in particular, positive braids are homogeneous). 

\begin{corollary}
    If $L$ is the closure of a homogeneous braid $\b\in B_n$, then $\ADO_p(L,x)$ has maximal degree and its top coefficient is $(-q)^{w(\b)-k_++k_-}$ where $w(\b)$ is the writhe of the braid and $k_{+}$ (resp. $k_-$) is the number of generators that appear with positive (resp. negative) sign.
\end{corollary}
\begin{proof}
 Let $F$ be the surface obtained by applying Seifert's algorithm to the braid. Then, as shown in \cite{Stallings:constructions-fibred}, $F$ can be obtained by plumbing $n_+=N_+-k_+$ positive Hopf bands and $n_-=N_--k_-$ negative Hopf bands, where $N_+$ (resp. $N_-$) is the number of positive (resp. negative) appearances of generators in $\b$. Since no deplumbing is needed, we have $m_+=m_-=0$ and hence $$2g(L)+s-1-2\la(\xi_F)=n_+-n_--m_++m_-=N_+-N_--k_++k_-=w(\b)-k_++k_-$$ so the corollary follows from our theorem.
\end{proof}

\begin{remark}
    \label{remark: positive braid closures}
    Note that positive braid closures are both fibred SQP and homogeneous, so both corollaries above apply: for $\b\in B_n$ the top coefficient of $\ADO_p(\widehat{\b})$ is $(-q)^{w-n+1}=(-q)^{2g+s-1}$ where $w$ is the writhe. Positive knots are not necessarily fibred (e.g. $5_2$), but if they are, then they are SQP by \cite{Rudolp:positive-are-SQ}, hence the top coefficient of their $\ADO_p$'s are as in Corollary \ref{corollary: positive braid and fibred SQ} above.
\end{remark}

\medskip

%\begin{corollary}\label{corollary: mutant knots}   If two mutant knots are fibred, then they must have the same genus and Hopf invariant.\end{corollary}\begin{proof}The genus statement follows from the fact that both knots have the same Alexander polynomial, which detects the genus of fibred knots. Now, the two knots also have the same colored Jones polynomials and hence the same ADO invariants by \cite{Willets:unification}. Our theorem implies that $\ADO_p$ detects the Hopf invariant mod $2p$. Since this is for all $p\geq 2$, the corollary follows.\end{proof}

 \section{Computations}\label{section: computations}

 \subsection{Knots with $\leq 7$ crossings} 

Our theorems are illustrated in Table \ref{table: N4 of small knots} where we computed $\ADO_4$ for small knots. Theorem \ref{Thm: main fibred} implies that, for fibred $K$, the top power of $x$ in $\ADO_4(K,x)$ is $x^{3g(K)}$ with coefficient $i^{g(K)-\la(\xi_F)}$ (recall $\la(\xi_F)=n_--m_-$ where $n_-$ (resp. $m_-$) is the number of negative Hopf bands plumbed (resp. deplumbed)). For non-fibred knots, we can still use the explicit expression of the top coefficient in Theorem \ref{Theorem: top coefficient of ADO after plumbing} to guess how they are obtained by plumbing Hopf bands on smaller (non-fibred) knots. For instance: %\textcolor{blue}{For fibred knots, these remarks are covered by the corollary ``ADO detects Hopf invariant". But remarks for non-fibered knots are more interesting, where EMN does not exist!}
\begin{enumerate}
  \item The $\ADO_4$'s of the knots $5_1, 7_1$ have top coefficient $q^{2g(K)}=i^{g(K)}$. Indeed, they are obtained by plumbing 4 and 6 positive Hopf bands. They are positive braid closures, so Remark \ref{remark: positive braid closures} applies.  %More generally, any link that is a plumbing of positive Hopf bands (for instance, all positive braid closures by \textcolor{blue}{It's a theorem of Stallings, ref?} will have top coefficient $q^{2g(L)+s-1}$. \textcolor{blue}{More generally, knots that are obtained by plumbing/deplumbing positive Hopf bands can be characterized as fibred strongly-quasipositive. Equivalently, they induce the tight contact structure of $S^3$.}
    \item The knots $6_3,7_7$ are plumbings of two positive and two negative Hopf bands. 
    \item The knots $6_2,7_6$ are both obtained by plumbing a figure-eight and a positive trefoil.
    \item The $\ADO_4$'s of the non-fibred knots $5_2,6_1,7_3,7_5$ have top powers which differ by multiplication by $\pm i$. We could indeed check that $7_3,7_5$ are obtained from $5_2$ by plumbing a positive trefoil and that $6_1$ can be obtained by plumbing two negative Hopf bands to $5_2$ (giving a top coefficient of $-i(-2+2i)$) and then deplumbing a figure-eight (which does not change the top coefficient).
   \item $7_2$ and $7_4$ cannot be obtained by plumbing/deplumbing Hopf bands on any of the other knots in the table.
\end{enumerate}

\begin{center}
\begin{table}
\begin{tabular}{ c|c|c }
 \hline
 $K$ & Fibred & $\ADO_{4}(K,x)$ \ (positive part)  \\ 
 \hline
 $3_1$ & $Y$&  $(1 + 2 i)+(1 + i) x + i x^2 + i x^3$  \\
  \hline
 $4_1$ & $Y$&$7+6 x + 3 x^2 + x^3$  \\
  \hline
 $5_1$ & $Y$ & $1+(1 + i) x + i x^2 - (1 - i) x^3 - (1 - i) x^4 - x^5 - x^6$\\
  \hline
 $5_2$ &$N$& $(-5 + 10 i)-(4 - 8 i) x - (3 - 5 i) x^2 - (2 - 2 i) x^3$     \\
  \hline
$6_1$&$N$&   $(7 + 10 i)+(4 + 8 i) x + (3 + 5 i) x^2 + (2 + 2 i) x^3$    \\
 \hline
$6_2$ &$Y$&  $(19 + 2 i)+(17 + 3 i) x + (11 + 6 i) x^2 + (6 + 8 i) x^3 $\\
& & $+ (2 + 6 i) x^4 + 
 3 i x^5 + i x^6$\\
  \hline
  $6_3$ & $Y$ & $41+36 x + 25 x^2 + 14 x^3 + 6 x^4 + 3 x^5 + x^6$ \\
  \hline
 $7_1$ &$Y$&  $(1-2i)-2 i x - i x^2 - x^4 - x^5 - (1 + i) x^6 - (1 + i) x^7 - i x^8 - i x^9$ \\
  \hline
 $7_2$ &$N$&  $(-5 + 4 i)-(4 - 2 i) x - (3 - 2 i) x^2 - (1 - 2 i) x^3$  \\ 
  \hline
 $7_3$ &$N$&  $(-5 + 16 i)-(7 - 15 i) x - (10 - 9 i) x^2 - (11 - 2 i) x^3$ \\
 
& & $- (9 + i) x^4 - (5 + 
    3 i) x^5 - (2 + 2 i) x^6$\\
     \hline
    $7_4$ &$N$&  $-35-28 x - 18 x^2 - 8 x^3$ \\
     \hline
     $7_5$ &$N$&  $(-37+34i)-(36 - 30 i) x - (34 - 18 i) x^2 - (28 - 6 i) x^3$ \\ 
   & &  $ - (17 + 
    i) x^4 - (8 + 4 i) x^5 - (2 + 2 i) x^6$ \\
     \hline
      $7_6$ &$Y$&  $(59 + 54 i)+(50 + 50 i) x + (32 + 41 i) x^2 + (15 + 29 i) x^3$ \\
     & & $+ (3 + 15 i) x^4 + 
 5 i x^5 + i x^6$ \\
     \hline
$7_7$ & $Y$& $(101-64i)+(91 - 55 i) x + (64 - 35 i) x^2 + (35 - 15 i) x^3$ \\ 
& & $+ (15 - 3 i) x^4 + 
 5 x^5 + x^6$\\ 
 \hline
\end{tabular}
\label{table: N4 of small knots}
\bigskip

\caption{Since $\ADO_p(K,x)$ is symmetric under $x\mapsto x^{-1}$, we only depict the nonnegative powers of $x$. }
\end{table}
\end{center}

Note that the Alexander polynomial does not give much information on the plumbing structure of knots. For instance, a fibred knot that is a plumbing of positive trefoils has $\ADO_2(K,x)$ with top coefficient $(-1)^g$. Both $7_1$ and $7_7$ satisfy the latter condition, but as we discussed above $7_7$ is not a trefoil plumbing while $7_1$ is.
%\textcolor{blue}{Trefoil plumbing or just positive Hopf band plumbing? Also, some of the above (e.g. (5)) can be deduced only with Alexander.}

\subsection{12-crossing knots} According to KnotInfo the ``false positives" of $\leq 12$ crossings where Alexander is monic and of maximal degree but that are not fibered are:
    \[12n_{57, 210, 214, 258, 279, 382, 394, 464, 483, 535, 650, 801, 815}\]
  In Table \ref{table: 12-crossing knots where Alex fails} we list the top term of $\ADO_p(K,x)$ of those knots at $p=3$ and $p=4$. Based on this, we note:
\begin{enumerate}
    \item \textbf{On $\ADO_3$}: The knots 12n57, 12n258, 12n279, 12n464, 12n483, 12n650, 12n815 have $\Z[q]$-monic $\ADO_3$. These knots also have $\Z[q]$-monic Links-Gould invariant (see Prop. 3.8 of \cite{Kohli:LG}). The former fact is indeed a consequence of the latter by \cite{Takenov:ADO3}, where it is shown that $\ADO_3$ is a specialization of Links-Gould for some knots (including those in the table).
    \item \textbf{On $\ADO_4$}: The non-fibredness of all the above knots, except $12n57$, is detected with $\ADO_4$. Note also that the knots 12n258, 12n394, 12n483, 12n535, 12n815 have top coefficients differing by a power of $i$. Similarly for the knots 12n279, 12n382, 12n464, 12n650. Though not appearing in the above table, all the $N_3$'s and $N_4$'s of the above knots are different except for the pair $12n210, 12n214$ which have the same $\ADO_p$ for all $p$ (since they are mutant).
\end{enumerate}

\begin{center}
\begin{table}
\begin{tabular}{ |c|c|c|c|c }
\hline
 $K$ & $g(K)$ & $\text{ top term of }\ADO_{3}(K,x)$ & $\text{ top term of }\ADO_{4}(K,x)$  \\ 
 \hline
 $12n57$ & 2 & $x^4$ & $x^6$\\
 \hline
  $12 n210$ & 3& $-2x^6$&   $- (1 + 4 i) x^9$\\
  \hline
  $12n214$ & 3& $-2x^6$ & $-(1+4i)x^9$\\
  \hline
      $12n258$ & 2& $q^4 x^4$ &   $(- 2 + i) x^6$ \\
      \hline
$12n279$ & 2   & $q^4 x^4$ &  $(1 - 2 i) x^6$\\
\hline
$12n382$ &2   &$2q x^4$ &  $(2 + i) x^6$\\
\hline
$12n394$ &2&  $2q x^4$& $(1 + 2 i) x^6$\\
\hline
 $12n464$ &2& $x^4$&  $(2 + i) x^6$\\
 \hline
$ 12n483$ &2   & $q^4 x^4$ & $(-2 + i) x^6$\\
 \hline
$12n535$ &2& $-2q^2 x^4$ & $(2 - i) x^6$\\
\hline
$12n650$ &2  & $q^4 x^4$& $(1-2i)x^6$\\
\hline
$12n801$ &2& $(-2+i\sqrt{3}) x^4$ &  $3 i x^6$\\
\hline
$12n815$ &2  & $q^4 x^4$ &    $(- 2 + i) x^6$\\
\hline
\end{tabular}
\label{table: 12-crossing knots where Alex fails}

\bigskip

\caption{Top terms of $\ADO_3$ and $\ADO_4$ of the 12-crossing knots whose non-fibredness is not detected by the Alexander polynomial.}
\end{table}
\end{center}

\bibliographystyle{amsplain}
\bibliography{referencesabr.bib}

\end{document}